\documentclass[reqno]{amsart}
\usepackage[utf8]{inputenc}
\usepackage{amsthm}
\usepackage{amsmath}
\usepackage{amsfonts}
\usepackage{amssymb}
\usepackage[svgnames]{xcolor}
\usepackage{colortbl}
\usepackage{iba-algo}

\newtheorem{theorem}{Theorem}[section]
\makeatletter
\newtheorem{lemma}{Lemma}
\renewcommand*{\c@lemma}{\c@theorem}
\renewcommand*{\p@lemma}{\p@theorem}

\renewcommand*{\c@conjecture}{\c@theorem}
\renewcommand*{\p@conjecture}{\p@theorem}

\newtheorem{proposition}{Proposition}
\renewcommand*{\c@proposition}{\c@theorem}
\renewcommand*{\p@proposition}{\p@theorem}

\newtheorem{corollary}{Corollary}
\renewcommand*{\c@corollary}{\c@theorem}
\renewcommand*{\p@corollary}{\p@theorem}

\renewcommand*{\c@observation}{\c@theorem}
\renewcommand*{\p@observation}{\p@theorem}

\theoremstyle{definition}
\newtheorem{problem}{Problem}
\renewcommand*{\c@problem}{\c@theorem}
\renewcommand*{\p@problem}{\p@theorem}

\renewcommand*{\c@definition}{\c@theorem}
\renewcommand*{\p@definition}{\p@theorem}

\newtheorem{remark}{Remark}
\renewcommand*{\c@remark}{\c@theorem}
\renewcommand*{\p@remark}{\p@theorem}

\renewcommand*{\c@example}{\c@theorem}
\renewcommand*{\p@example}{\p@theorem}

\newtheorem{algorithm}{Algorithm}
\renewcommand*{\c@algorithm}{\c@theorem}
\renewcommand*{\p@algorithm}{\p@theorem}

\makeatother

\numberwithin{equation}{section}

\makeatletter
\let\bfseries=\undefined
\DeclareRobustCommand\bfseries
{\not@math@alphabet\bfseries\mathbf
  \boldmath\fontseries\bfdefault\selectfont}
\makeatother

\newcommand{\Q}{{\mathbb Q}}
\newcommand{\Z}{{\mathbb Z}}

\newcommand{\R}{\mathbb R}

\def\IP{({\rm IP})_{N,\veb,\vel,\veu,f}}
\def\CP{({\rm CP})_{N,\veb,\vel,\veu,f}} % its continuous relaxation
\def\Graver{{\mathcal G}}

\def\Orthant_j{{\mathcal O}_{j}}
\def\Circuits{{\mathcal C}}

\usepackage{enumerate}
\def\ve#1{\mathchoice{\mbox{\boldmath$\displaystyle\bf#1$}}
{\mbox{\boldmath$\textstyle\bf#1$}}
{\mbox{\boldmath$\scriptstyle\bf#1$}}
{\mbox{\boldmath$\scriptscriptstyle\bf#1$}}}
\newcommand\vea{{\ve a}}
\newcommand\veb{{\ve b}}
\newcommand\vecc{{\ve c}}
\newcommand\ved{{\ve d}}

\newcommand\veg{{\ve g}}

\newcommand\vel{{\ve l}}

\newcommand\ver{{\ve r}}

\newcommand\veu{{\ve u}}
\newcommand\vev{{\ve v}}

\newcommand\vex{{\ve x}}
\newcommand\vey{{\ve y}}
\newcommand\vez{{\ve z}}

\newcommand\vebar[1]{\ve{\bar #1}}
\newcommand\vebarx{\vebar x}
\newcommand\vebary{\vebar y}

\newcommand\velambda{\ve{\lambda}}

\newcommand{\maO}{O}            % zero matrix

\usepackage{ifthen}
\newcommand{\DeclareBracket}[3]{
  \newcommand{#1}[2][]{%
  \ifthenelse%
  {\equal{##1}{}}%
  {\left#2##2\right#3}%
  {\csname ##1l\endcsname#2##2\csname ##1r\endcsname#3}}}
\DeclareBracket\set\{\}
\DeclareBracket\floor\lfloor\rfloor
\DeclareBracket\ceil\lceil\rceil
\DeclareBracket\bracket[]
\DeclareBracket\paren()

  \newcommand\Ye{\cellcolor{yellow}}
  \newcommand\Gr{\cellcolor{LimeGreen}}
  \newcommand\Bl{\cellcolor{CornflowerBlue}}
  \newcommand\Ro{\cellcolor{Coral}}

\newenvironment{psmallmatrix}{\left(\smallmatrix}{\endsmallmatrix\right)}

\newcommand\FourBlockBig[5][\relax]{\begin{pmatrix}#2& #3\\#4&#5 \end{pmatrix}\ifx#1\relax\else^{(#1)}\fi}
\newcommand\FourBlock[5][\relax]{\begin{psmallmatrix}#2& #3\\#4&#5 \end{psmallmatrix}\ifx#1\relax\else{^{(#1)}}\fi}
\newcommand\NoBlock{\cdot}      % the 0x0 matrix.  (As an alternative, \circ
\newcommand\BigO{\mathrm{O}}

\usepackage{hyperref}

\begin{document}
\pagestyle{headings}  % switches on printing of running heads

\title[Graver basis and proximity techniques]{%% A polynomial-time algorithm for the
%%   $N$-fold $4$-block decomposable integer separable convex optimization
%%   problem
  Graver basis and proximity techniques
  for  block-structured separable convex\\ integer minimization problems}
\author{Raymond Hemmecke}
\address{Raymond Hemmecke: 
Zentrum Mathematik, M9, Technische Universit\"at M\"unchen, Boltzmannstr.\ 3, 
85747 Garching, Germany}
\email{hemmecke@ma.tum.de}
\author{Matthias K\"oppe}
\address{Matthias K\"oppe: Department of
  Mathematics, University of California,
  Davis, One Shields Avenue, Davis, CA, 95616, USA}
\email{mkoeppe@math.ucdavis.edu}
\author{Robert Weismantel}
\address{Robert Weismantel: Institute for Operations Research, 
ETH Z\"urich, R\"amistrasse 101, 8092 Zurich, Switzerland}
\email{robert.weismantel@ifor.math.ethz.ch}

\date{$\relax$Revision: 83 $ - \ $Date: 2012-07-04 17:48:28 -0700 (Wed, 04 Jul 2012) $ $}

\dedicatory{Dedicated to the memory of Uri Rothblum}

\maketitle

\begin{abstract}
  We consider $N$-fold $4$-block decomposable integer programs, which
  simultaneously generalize $N$-fold integer programs and two-stage stochastic
  integer programs with $N$ scenarios.  In previous work [R.~Hemmecke,
  M.~K{\"o}ppe, R.~Weismantel, \emph{A polynomial-time algorithm for
    optimizing over {$N$}-fold 4-block decomposable integer programs}, Proc.\
  IPCO 2010, Lecture Notes in Computer Science, vol. 6080, Springer, 2010,
  pp.~219--229], it was proved that for fixed blocks but variable~$N$, these
  integer programs are polynomial-time solvable for any linear objective.  We
  extend this result to the minimization of separable convex objective
  functions.  Our algorithm combines Graver basis techniques with a proximity
  result [D.S.~Hochbaum and J.G.~Shanthikumar, \emph{Convex separable
    optimization is not much harder than linear optimization}, J. ACM
  \textbf{37} (1990), 843--862], which allows us to use convex continuous
  optimization as a subroutine.
  
\end{abstract}

%\newpage
\section{Introduction}

We consider a family of  nonlinear integer minimization
problems over block-structured linear constraint systems
in variable dimension. The objective is to  minimize a  separable
convex objective function $f\colon \R^n \to \R$, defined as 
$$f(x_1,\ldots,x_n)= \sum_{i=1}^n f_i(x_i),$$ with convex functions $f_i\colon \R \to \R$
of one variable each. %% We assume that we have access to an oracle, that
%% queried on $x \in \Z^n$ returns the values $f_i(x_i)$ for all $i$. This then
%% allows us to compute $f(x)$, of course.  \texttt{(This is not consistent with
%%   the discussion later... --Matthias.)}
%% Some specific families of integer minimization problems defined by linear constraints and separable convex objective function have been studied in the literature. 

Hochbaum and Shanthikumar \cite{hochbaum-shanthikumar:1990:convex-separable} present a general technique for transforming algorithms for linear integer minimization to algorithms for separable convex integer minimization. The key ingredients  of this transformation technique are scaling techniques and proximity results between optimal integer solutions and optimal solutions of the continous relaxation.
This technique leads directly to polynomial time algorithms if all the
subdeterminants of the constraint matrix are bounded polynomially.  

Of course, this is quite a restrictive hypothesis, but 
an important corollary of this work is a polynomial time algorithm for minimizing a separable convex function over systems of inequalities associated with a unimodular matrix. This generalizes, in particular,  earlier work of Minoux \cite{Minoux1986} on minimum cost flows with separable convex cost functions.     
 
An impossibility result on the existence of a strongly polynomial algorithm for minimizing a general separable convex function over network flow constraints has been shown in \cite{Hochbaum1994}. \smallbreak

In the present paper, we study a certain family of block-structured separable
convex integer minimization problems over polyhedra, which does not satisfy
the hypothesis of polynomially bounded subdeterminants.  The constraint matrix
of these problems is
\emph{$N$-fold $4$-block decomposable} as follows:
\[
\FourBlockBig[N]CDBA :=
\begin{pmatrix}
C & D & D & \cdots & D \\
B & A & \maO  &   & \maO  \\
B & \maO  & A &   & \maO  \\
\vdots &   &   & \ddots &   \\
B & \maO  & \maO  &   & A
\end{pmatrix}
\]
for some given $N\in\Z_+$ and $N$ copies of $A$, $B$, and $D$. This problem
type was studied recently 
in~\cite{hemmecke-koeppe-weismantel:4-block}. 

$N$-fold $4$-block decomposable matrices arise in many contexts and have been
studied in various special cases, three of which are particularly relevant.  
We denote by $O$ a zero matrix of compatible dimensions and by $\NoBlock$ a
matrix with no columns or no rows. 

(i) For $C=\NoBlock $ and $D=\NoBlock $ we recover the problem matrix
$\FourBlock[N]\NoBlock\NoBlock{B}{A}$ of a two-stage
stochastic integer optimization problem. Then, $B$ is the matrix associated with the first stage decision variables and $A$ is associated with the decision on stage~2. The number of occurences of blocks of the matrix $A$ reflect all the possible scenarios that pop up once a first stage decision has been made.  We refer to \cite{Hemmecke:SIP2} for a survey on state of the art techniques to solve this problem.

(ii) For $B=\NoBlock $ and $C=\NoBlock $ we recover the problem
matrix~$\FourBlock[N]\NoBlock{D}\NoBlock{A}$ of a so-called $N$-fold integer
problem.  
Here, if we let $A$ be the node-edge incidence matrix of the given network and set $D$ to be the identity, then the resulting $N$-fold IP is a multicommodity network flow problem. Separable convex $N$-fold IPs can be solved in polynomial time, provided that the matrices $A$ and $D$ are fixed \cite{DeLoera+Hemmecke+Onn+Weismantel:08,Hemmecke+Onn+Weismantel:08}.

(iii) For totally unimodular matrices $C,A$ their so-called $1$-sum
$\FourBlock{C}{O}{O}{A}$ is totally unimodular. 
Similarly, total unimodularity is preserved under the so-called $2$-sum and $3$-sum composition \cite{Schrijver86,Seymour80}.   For example, for matrices $C$ and $A$, column vector $\vea$ and row vector
$\veb^\intercal$ of appropriate dimensions, the $2$-sum of
$\left(\begin{smallmatrix}C&\vea\\\end{smallmatrix}\right)$ and
$\left(\begin{smallmatrix}\veb^\intercal\\A\\\end{smallmatrix}\right)$ gives
$\left(\begin{smallmatrix}C&\vea\veb^\intercal\\
\maO &A\\\end{smallmatrix}\right)$. The \mbox{$2$-sum} of
$\left(\begin{smallmatrix}C&\vea\veb^\intercal & \ve a\\
\maO &A&\maO \\\end{smallmatrix}\right)$ and
$\left(\begin{smallmatrix}\veb^\intercal\\B\\\end{smallmatrix}\right)$ creates
the matrix
$$\left(\begin{smallmatrix}C&\vea\veb^\intercal&\vea\veb^\intercal\\\maO &
A&\maO \\\maO &\maO &A\\\end{smallmatrix}\right),$$ which is the $2$-fold
\mbox{$4$-block} decomposable matrix
$\left(\begin{smallmatrix}C&\vea\veb^\intercal\\
\maO &A\\\end{smallmatrix}\right)^{(2)}$.   Repeated application of certain $1$-sum, $2$-sum and \mbox{$3$-sum} compositions leads to a particular family of $N$-fold $4$-block decomposable matrices with special structure regarding the matrices $B$ and $D$.

\begin{figure}[t]
  \centering\small
  $\begin{array}{cccc@{\ }cccc@{\ }cccc@{\ }cc@{\ }lp{.3\linewidth}}
          \ve x^{\mathrm{agg}} && \ve x^1 & & & &  \ve x^2 & & & & \ve x^M
          & & & & & \textit{(Flows)} \\
          \\[-1ex]
          \Bl -I && \Gr I & \Gr &
          \Gr  & & \Gr  I &
          \Gr  & \Gr  & \cdots & \Gr  I & \Gr & \Gr & = & \ve 0 & Flow aggregation
          \\
          \\[-1.5ex]
          \Ro && \Ye A & \Ye  & \Ye  & &&&&&&&& = & \ve b_1 & Flow commodity 1 \\
          \Ro I && \Ye  & \Ye I & \Ye -I & &&&&&&&& = & \ve u_1 & Capacity scenario 1\\
          \\[-1.5ex]
          %\hline
          \Ro && &&&& \Ye A & \Ye  & \Ye  & &&&& = & \ve b_2 & Flow commodity 2\\
          \Ro I && &&&& \Ye  & \Ye I &\Ye -I & &&&& = & \ve u_2 & Capacity scenario 2\\
          \vdots & &&&& &&&& \ddots\\
          \Ro & &&&& &&&& & \Ye A & \Ye  & \Ye  & = & \ve b_M & Flow commodity $M$ \\
          \Ro I & &&&& &&&& & \Ye & \Ye I & \Ye -I & = & \ve u_N & Capacity
          scenario $N$\\
          \\[-1ex]
          &&& \ve s_1 & \ve t_1 &&&\ve s_2 & \ve t_2 &&& \ve s_N & \ve t_N &
          & & \textit{(Slack/Excess)}
        \end{array}$
  \caption{Modeling a two-stage stochastic integer multi-commodity flow problem as an $N$-fold $4$-block decomposable problem.  Without loss of generality, the number of commodities and the number of scenarios are assumed to be equal.}
  \label{fig:2stage-multicommodity}
\end{figure}

(iv) The general case appears in 
stochastic integer programs with second order dominance relations
\cite{Gollmer+Gotzes+Schultz} and 
stochastic integer multi-commodity flows.
See \cite{hemmecke-koeppe-weismantel:4-block} for further details of the model
as an $N$-fold $4$-block decomposable problem. 
To give one example consider a
stochastic integer multi-commodity flow problem, introduced in
\cite{Mirchandani+Soroush,Powell+Topaloglu}. Let  $M$ integer (in contrast to
continuous) commodities to be transported over a given network. While we
assume that supply and demands are deterministic, we assume that the upper
bounds for the capacities per edge are uncertain and given initially only via
some probability distribution. In a first stage we have to decide how to
transport the $M$ commodities over the given network without knowing the true
capacities per edge. Then, after observing the true capacities per edge,
penalties have to be paid if the capacity is exceeded. Assuming that we have
knowledge about the probability distributions of the uncertain upper bounds,
we wish to minimize the costs for the integer multi-commodity flow plus the
expected penalties to be paid for exceeding capacities. To solve this problem,
we discretize as usual the probability distribution for the uncertain upper
bounds into $N$ scenarios. Doing so, we obtain a (typically large-scale)
(two-stage stochastic) integer programming problem as shown in Figure~\ref{fig:2stage-multicommodity}.
Herein, $A$ is the node-edge incidence matrix of the given network, $I$ is an
identity matrix of appropriate size, and the columns containing $-I$
correspond to the penalty variables. 

%%%%%%%%%%%%%%%%%%%%%%%%%%%%%%%%%%%%%%%%%%%%%%%%%%%%%%%%%%%%%%%%%%%%%%%%%%%%%%%%%%%%%%%%%%%%%%%%%%%%%%%%%%%%%%%%%%%%%%%

\section{Main results and proof outline}

In \cite{hemmecke-koeppe-weismantel:4-block}, the authors proved the following
result. 
\begin{theorem}\label{Theorem: ABCD}
Let $A\in\Z^{d_A\times n_A}$, $B\in\Z^{d_A\times n_B}$, $C\in\Z^{d_C\times
  n_B}$, $D\in\Z^{d_C\times n_A}$ be fixed matrices. For given $N\in\Z_+$ let
$\vel\in(\Z\cup\{-\infty\})^{n_B+Nn_A}$,
$\veu\in(\Z\cup\{+\infty\})^{n_B+Nn_A}$, $\veb\in\Z^{d_C+Nd_A}$, and let 
$f\colon\R^{n_B+Nn_A}\rightarrow\R$ be a separable convex function 
that takes integer values on~$\Z^{n_B+Nn_A}$
and denote by
$\hat f$ an upper bound on the maximum of $|f|$ over the feasible region of
the $N$-fold 
$4$-block decomposable convex integer minimization problem
\[
\IP\colon\qquad\min\left\{f(\vez):\FourBlock[N]{C}{D}{B}{A}\vez=\veb,\;\vel\leq\vez\leq\veu,\;\vez\in\Z^{n_B+Nn_A}\right\}.%
\]%
We assume that $f$ is given only by a comparison oracle that, when queried on $\vez$ and $\vez'$ decides whether $f(\vez)<f(\vez')$, $f(\vez)=f(\vez')$ or $f(\vez)>f(\vez')$. Then the following holds:
\begin{enumerate}[\rm(a)]
\item There exists an algorithm with input $N$, $\vel,\veu$, $\veb$ that
  computes a feasible solution to $\IP$ or 
  decides that no such solution exists and that runs in time polynomial in $N$
  and in the binary encoding lengths $\langle \vel,\veu,\veb\rangle$.
\item There exists an
  algorithm  with input $N$, $\vel,\veu$, $\veb$ and a feasible solution
  $\vez_0$ to $\IP$ that decides whether $\vez_0$ is optimal or finds a better
  feasible solution $\vez_1$ to the problem~$\IP$ with $f(\vez_1)<f(\vez_0)$
  and that runs in time polynomial in $N$ and in the binary encoding lengths
  $\langle \vel,\veu,\veb,\hat f\rangle$. 
\item For the restricted problem where $f$ is linear, 
  there exists an algorithm with input $N$, $\vel,\veu$, $\veb$ that finds an optimal
  solution to the problem $\IP$ or decides that $\IP$ is infeasible or unbounded and
  that runs in time polynomial in $N$ and in the binary encoding lengths
  $\langle \vel,\veu,\veb,\hat f\rangle$. 
\end{enumerate}
\end{theorem}

This theorem generalizes a similar statement for $N$-fold integer programming
and for two-stage stochastic integer programming. In these two special cases,
one can even prove claim (c) of Theorem \ref{Theorem: ABCD} for all separable
convex functions and for a certain class of separable convex functions,
%% \texttt{(What is that ``certain class''?? --Matthias)},
respectively.  
In~\cite{hemmecke-koeppe-weismantel:4-block}, it was posed as an open question
whether Theorem \ref{Theorem: ABCD} can be extended, for the full class of
$N$-fold $4$-block decomposable problems, from linear $f$ to general separable
convex functions~$f$.

In the present paper, we settle this question, proving the following result
for separable convex functions~$f$, for which we assume that the following
approximate continuous convex optimization oracle is available:
\begin{problem}[Approximate continuous convex optimization]
  Given the data $A$, $B$, $C$, $D$, $N$, $\vel,\veu$, $\veb$ and a number
  $\epsilon\in\Q_{>0}$, 
  find a feasible solution $\ve{r}_\epsilon \in \Q^{n_B+Nn_A}$ for the continuous
  relaxation
  \[
  \CP\colon\qquad\min\left\{f(\ver):\FourBlock[N]{C}{D}{B}{A}\ver=\veb,\;\vel\leq\ver\leq\veu,\;\ver\in\R^{n_B+Nn_A}\right\}.%
  \]%
  such that there exists an optimal solution $\ve{\hat r}$ to $\CP$ with
  \begin{displaymath}
    \|\ve{\hat r} - \ve{r}_\epsilon\|_\infty \leq \epsilon,
  \end{displaymath}
  or report \textsc{Infeasible} or \textsc{Unbounded}. 
\end{problem}

\begin{theorem}\label{Theorem: ABCD convex}
  For the problem of Theorem~\ref{Theorem: ABCD}, 
  we assume that the objective function~$f$ is given by an evaluation oracle
  and an approximate continuous convex optimization oracle for $\CP$.
  
  Then there exists an algorithm
  that finds an optimal solution to $\IP$ or decides that $\IP$ is infeasible
  or unbounded and that runs in time polynomial in $N$ and in the binary encoding
  lengths $\langle \vel,\veu,\veb,\hat f\rangle$.
\end{theorem}

The main new technical contribution of the present paper is to combine Graver
basis techniques with a proximity result developed by Hochbaum and
Shanthikumar~\cite{hochbaum-shanthikumar:1990:convex-separable} in the context
of their so-called proximity-scaling technique.

This allows us to first use the approximate continuous convex optimization
oracle to find 
a point, in whose proximity the optimal integer solution has to lie.
The integer problem restricted to this neighborhood is then 
efficiently solvable with primal (augmentation) algorithms using Graver
bases, which will find 
the optimal integer solution in a polynomial number of steps.

\bigbreak

We now briefly explain the Graver basis techniques; we refer the reader to the
survey paper \cite{Onn:survey} or the monograph~\cite{onn:nonlinear-discrete-monograph} for more details.
  Let $E\in\Z^{d\times n}$
be a matrix. We associate with $E$ a finite set $\Graver(E)$ of vectors with
remarkable properties. Consider the set $\ker(E)\cap\Z^n$. Then we put into
$\Graver(E)$ all nonzero vectors $\vev\in\ker(E)\cap\Z^n$ that cannot be
written as a sum $\vev=\vev'+\vev''$ of nonzero vectors
$\vev',\vev''\in\ker(E)\cap\Z^n$ that lie in the same orthant (or
equivalently, have the same sign pattern in $\{\geq 0,\leq 0\}^n$) as
$\vev$. The set $\Graver(E)$ has been named the \emph{Graver basis} of $E$,
since Graver \cite{Graver75} introduced this set $\Graver(E)$ in $1975$ and
showed that it constitutes an optimality certificate (test set) for the family
of integer linear programs that share the same problem matrix,~$E$. By this we
mean that $\Graver(E)$ provides an augmenting vector for any non-optimal
feasible solution and hence allows the design of a simple augmentation
algorithm to solve the integer linear program in a finite number of
augmentations. 

The augmentation technique can also be used to efficiently construct a
feasible solution in the first place, in a procedure similar to phase~I of the
simplex algorithm~\cite{Hemmecke:2003b}.

More recently, it has been shown in \cite{MurotaSaitoWeismantel04} that
$\Graver(E)$ constitutes an optimality 
certificate for a wider class of integer minimization problems, namely for
those minimizing a % concave or a
separable convex objective function over
a feasible region of the form
$$\{\,\vez:E\vez=\veb,\;\vel\leq\vez\leq\veu,\;\vez\in\Z^n\,\}.$$

Moreover, several techniques have been found to turn the augmentation
algorithm into an efficient algorithm, bounding the number of augmentation
steps polynomially.  Three such speed-up techniques are known in the literature:  
For 0/1 integer linear problems, a simple \emph{bit-scaling technique} suffices
\cite{SchulzWeismantelZiegler95}.  
For general integer linear problems, one can use the \emph{directed augmentation
technique} \cite{SchulzWeismantel99}, in which one uses Graver basis
elements~$\ve v\in\Graver(E)$ that are improving directions for the nonlinear functions $\ve
c^\intercal \ve v^+ + \ve d^\intercal \ve v^-$, which are adjusted during the
augmentation algorithm.  
For separable convex integer problems, one can use the \emph{Graver-best
augmentation technique} \cite{Hemmecke+Onn+Weismantel:08}, where one uses
an augmentation vector~$\ve v$  
that is at least as good as the best  
augmentation step of the form $\gamma\ve g$ with $\gamma\in\Z_+$ and
$\veg\in\Graver(E)$.\smallbreak

In~\cite{hemmecke-koeppe-weismantel:4-block}, the authors found a way to
implement the directed augmentation technique efficiently for 
$N$-fold $4$-block decomposable integer programs, despite the
exponential size of the Graver basis.  This gives an efficient optimization
algorithm for the case of linear objective functions,
proving~\autoref{Theorem: ABCD}.  
It is still an open question whether the Graver-best augmentation technique
can be implemented efficiently.  This would give an alternative proof of
\autoref{Theorem: ABCD convex}.

The paper~\cite{hemmecke-koeppe-weismantel:4-block} and the present paper
crucially rely on the following structural result about
$\Graver\left(\FourBlock[N]{C}{D}{B}{A}\right)$, which was proved
in~\cite{hemmecke-koeppe-weismantel:4-block}.
\begin{theorem}\label{Theorem: Polynomial bound for fixed ABCD on increase of Graver degree}
If $A\in\Z^{d_A\times n_A}$, $B\in\Z^{d_A\times n_B}$, $C\in\Z^{d_C\times n_B}$, $D\in\Z^{d_C\times n_A}$ are fixed matrices, then $\max\left\{\,\|\vev\|_1:\vev\in\Graver\left(\FourBlock[N]{C}{D}{B}{A}\right)\,\right\}$ is bounded by a polynomial in $N$.
\end{theorem}

We note that in the special case of $N$-fold IPs, the $\ell_1$-norm is bounded
by a constant (depending only on the fixed problem matrices and not on $N$),
and in the special case of two-stage stochastic IPs, the $\ell_1$-norm is
bounded linearly in~$N$.  This fact demonstrates that $N$-fold
\mbox{$4$-block} IPs are much richer and more difficult to solve than the two
special cases.

\section{Proof of the results}

\subsection{Aggregation technique}

We will use an aggregation/disaggregation technique, which is based on the
following folklore fact on Graver bases (see, for example, Corollary 3.2 in
\cite{Hemmecke:Z-convex}): 
\begin{lemma}[Aggregation] \label{Lemma: Graver(A a a) from Graver(A a)}
  Let $G=(F\ \ve f\ \ve f)$ be a matrix with two identical columns. 
  Then the Graver bases of $(F\ \ve f)$ and $G$ are related as follows:
  \[
  \Graver(G)=\{\,(\veu,v,w):vw\geq 0, (\veu,v+w)\in\Graver((F\ \ve f))\,\}\cup\{\pm (\ve 0,1,-1)\}.
  \]
\end{lemma}
Thus, the maximum $\ell_1$-norm of Graver basis elements does not change if we repeat columns.
\begin{corollary}\label{Corollary: 1-norm unchanged by aggregation}
  Let $G$ be a matrix obtained from a matrix~$F$ by repeating columns.
  Then 
  \begin{displaymath}
    \max\{\,\|\vev\|_1:\vev\in\Graver(G)\,\} = \max\{2,
    \max\{\,\|\vev\|_1:\vev\in\Graver(F)\,\}\}. 
  \end{displaymath}
\end{corollary}

\subsection{Bounds for Graver basis elements}
Let us start by bounding the $\ell_1$-norm of Graver basis elements of matrices.
The following result can be found, for instance, in~\cite[Lemma 3.20]{onn:nonlinear-discrete-monograph}.
\begin{lemma}[Determinant bound]\label{Lemma: Determinant bound}
Let $A\in\Z^{m\times n}$ be a matrix of rank~$r$ and let $\Delta(A)$ denote
the maximum absolute value of subdeterminants of~$A$.
Then $\max\{\,\|\vev\|_1:\vev\in\Graver(A)\,\} \leq (n-r)(r+1) \Delta(A)$.
Moreover, $\Delta(A) \leq (\sqrt{m} M)^m$, where $M$ is the maximum absolute
value of an entry of~$A$.
\end{lemma}

As a corollary of Lemma~\ref{Lemma: Determinant bound} and the aggregation
technique (Corollary~\ref{Corollary: 1-norm unchanged by aggregation}), we
obtain the following result:
\begin{corollary}[Determinant bound, aggregated]\label{Corollary: Determinant bound, aggregated}
Let $A\in\Z^{m\times n}$ be a matrix of rank~$r$ and let $d$ be the number of
different columns in~$A$ and $M$ the maximum absolute
value of an entry of~$A$.
Then $$\max\{\,\|\vev\|_1:\vev\in\Graver(A)\,\} \leq (d-r)(r+1) (\sqrt{m} M)^m.$$
\end{corollary}

For matrices with only one row ($m = r= 1$), there are only $2M+1$ different columns, and
so this bound simplifies to~$4M^2$.  However, a tighter bound is known for
this special case.  The following lemma is a straight-forward consequence of
Theorem 2 in \cite{DiaconisGrahamSturmfels95}. 
\begin{lemma}[PPI bound]\label{Lemma: PPI degree bound}
Let $A\in\Z^{1\times n}$ be a matrix consisting of only one row and let $M$ be an upper bound on the absolute values of the entries of $A$. Then we have $\max\{\,\|\vev\|_1:\vev\in\Graver(A)\,\}\leq 2M-1$.
\end{lemma}

Let us now prove some more general degree bounds on Graver bases that we will use in the proof of the main theorem below.

\begin{lemma}[Graver basis length bound for stacked matrices]
  \label{Lemma: General bound on increase of Graver degree}
  Let $L\in\Z^{d\times n}$ and let $F\in\Z^{m\times n}$. Moreover, put $E:=\left(\begin{smallmatrix}F\\L\\\end{smallmatrix}\right)$. Then we have
\[
\max\{\|\vev\|_1:\vev\in\Graver(E)\}\leq\max\{\|\ve\lambda\|_1:\ve\lambda\in \Graver(F\cdot\Graver(L))\}\cdot\max\{\|\vev\|_1:\vev\in\Graver(L)\}.
\]
%% Mnemonik:  F ist der obere Teil von E, L ist der untere Teil von E.
\end{lemma}

\begin{proof}
Let $\vev\in\Graver(E)$. Then $\vev\in\ker(L)$ implies that $\vev$ can be written as a nonnegative integer linear sign-compatible sum $\vev=\sum\lambda_i \veg_i$ using Graver basis vectors $\veg_i\in\Graver(L)$. Adding zero components if necessary, we can write \mbox{$\vev=\Graver(L)\ve\lambda$}. We now claim that $\vev\in\Graver(E)$ implies $\ve\lambda\in\Graver(F\cdot\Graver(L))$.

First, observe that $\vev\in\ker(F)$ implies $F\vev=F\cdot(\Graver(L)\ve\lambda)=(F\cdot\Graver(L))\ve\lambda=\ve 0$ and thus, $\ve\lambda\in\ker(F\cdot\Graver(L))$. If $\ve\lambda\not\in\Graver(F\cdot\Graver(L))$, then it can be written as a sign-compatible sum $\ve\lambda=\ve\mu+\ve\nu$ with $\ve\mu,\ve\nu\in\ker(F\cdot\Graver(L))$. But then
\[
\vev=(\Graver(L)\ve\mu)+(\Graver(L)\ve\nu)
\]
gives a sign-compatible decomposition of $\vev$ into vectors
$\Graver(L)\ve\mu,\Graver(L)\ve\nu\in\ker(E)$, contradicting the minimality
property of $\vev\in\Graver(E)$. Hence,
$\ve\lambda\in\Graver(F\cdot\Graver(L))$.

From $\vev=\sum\lambda_i \veg_i$ with $\veg_i\in\Graver(L)$ and
$\ve\lambda\in\Graver(F\cdot\Graver(L))$, the desired estimate follows. 
\end{proof}

We will employ the following simple corollary.

\begin{corollary}\label{Corollary: Special bound on increase of Graver degree}
Let $L\in\Z^{d\times n}$ and let $\ve a^\intercal\in\Z^n$ be a row
vector. Moreover, put $E:=\left(\begin{smallmatrix}\ve a^\intercal \\L \\\end{smallmatrix}\right)$. Then we have
\[
\max\{\|\vev\|_1:\vev\in\Graver(E)\}\leq\left(2\cdot\max\left\{|\ve a^\intercal\vev|:\vev\in\Graver(L)\right\}-1\right)\cdot\max\{\|\vev\|_1:\vev\in\Graver(L)\}.
\]
In particular, if $M:=\max\{|a^{(i)}|:i=1,\ldots,n\}$ then
\[
\max\{\|\vev\|_1:\vev\in\Graver(E)\}\leq 2nM\left(\max\{\|\vev\|_1:\vev\in\Graver(L)\}\right)^2.
\]
\end{corollary}

\begin{proof}
By Lemma \ref{Lemma: General bound on increase of Graver degree}, we already get
\[
\max\{\|\vev\|_1:\vev\in\Graver(E)\}\leq\max\{\|\ve\lambda\|_1:\ve\lambda\in
\Graver(\ve a^\intercal\cdot\Graver(L))\}\cdot\max\{\|\vev\|_1:\vev\in\Graver(L)\}.
\]
Now, observe that $\ve a^\intercal\cdot\Graver(L)$ is a $1\times|\Graver(L)|$-matrix. Thus, the degree bound of primitive partition identities, Lemma \ref{Lemma: PPI degree bound}, applies, which gives
\[
\max\{\|\ve\lambda\|_1:\ve\lambda\in \Graver(\ve
a^\intercal\cdot\Graver(L))\}\leq 2\cdot\max\left\{|\ve a^\intercal\vev|:\vev\in\Graver(L)\right\}-1,
\]
and thus, the first claim is proved. The second claim is a trivial consequence
of the first. 
\end{proof}

Let us now extend this corollary to a form that we need to prove
\autoref{Theorem: Polynomial bound for fixed ABCD on increase of Graver
  degree}. 

\begin{corollary}\label{Corollary: Recursive bound on increase of Graver degree}
Let $L\in\Z^{d\times n}$ and let $F\in\Z^{m\times n}$. Let the entries of $F$ be bounded by $M$ in absolute value. Moreover, put $E:=\left(\begin{smallmatrix}F\\L\\\end{smallmatrix}\right)$. Then we have
\[
  \max\{\|\vev\|_1:\vev\in\Graver(E)\}\leq (2nM)^{2^m-1}\left(\max\{\|\vev\|_1:\vev\in\Graver(L)\}\right)^{2^m}.
\]
\end{corollary}

\begin{proof} This claim follows by simple induction, adding one row of $F$ at
  a time, and by using the second inequality of Corollary \ref{Corollary:
    Special bound on increase of Graver degree} to bound the sizes of the
  intermediate Graver bases in comparison to the Graver basis of the matrix
  with one row of $F$ fewer. 
\end{proof}

In order to give a proof of \autoref{Theorem: Polynomial bound for fixed ABCD on increase of Graver degree}, let us consider the submatrix $\FourBlock[N]\NoBlock\NoBlock BA$. A main result from \cite{Hemmecke:SIP2} is the following.

\begin{theorem}[Graver basis for stochastic IPs]
  \label{Lemma: Finite bound for SIP Graver bases}
  Let $A\in\Z^{d_A\times n_A}$ and $B\in\Z^{d_A\times n_B}$, 
  and let $\Graver =
  \Graver\left(\FourBlock[N]\NoBlock\NoBlock BA\right)$. 
  There exist numbers $g, \xi, \eta\in\Z_+$ depending only on $A$ and $B$ but not
  on $N$ such that the following holds:
  \begin{enumerate}[\rm(a)]
  \item  For every $N\in\Z_+$ and for every
    $\vev\in\Graver$, we have $\|\ve v\|_\infty \leq g$, i.e., 
    the components of $\vev$ are bounded by $g$ in absolute value.
  \item As a corollary, $\|\vev\|_1\leq (n_B+Nn_A)g$ for all
    $\vev\in\Graver$.
  \item More precisely, there exists a finite set $X\subseteq\Z^{n_B}$ of
    cardinality $|X|\leq \xi$ and for each $\ve x \in X$ a finite set $Y_{\ve
      x}\subseteq\Z^{n_A}$ of cardinality $|Y_{\ve x}|\leq \eta$ such that the
    elements $\ve v\in \Graver$ take the form $\ve v= (\ve x, \ve y^1, \dots,
    \ve y^n)$, with $\ve x\in X$ and $\ve y^1,\dots,\ve y^n\in Y_{\ve x}$.
  \end{enumerate}
\end{theorem}
\begin{remark}\label{rem:stoch-bounds-effective}
  The finiteness of the numbers $g,\xi,\eta$ comes from a saturation result in
  commutative algebra.  Concrete bounds on these numbers are unfortunately not
  available.  However, for given matrices $A$ and $B$, the finite sets $X$ and
  $Y_{\ve x}$ for $\ve x \in X$ can be computed using the Buchberger-type
  completion algorithm in \cite[section~3.3]{Hemmecke:SIP2}.  Thus, the
  numbers $g,\xi,\eta$ are effectively computable.
\end{remark}

Combining this result with Corollary \ref{Corollary: Recursive bound on increase of Graver degree}, we get a bound for the $\ell_1$-norms of the Graver basis elements of $\FourBlock[N]{C}{D}{B}{A}$. %% Note that the second claim of the following corollary is exactly Theorem \ref{Theorem: Polynomial bound for fixed ABCD on increase of Graver degree}.

\begin{proposition}[Graver basis length bound for 4-block IPs]\label{Proposition: Polynomial bound for fixed ABCD on increase of Graver degree}
Let $A\in\Z^{d_A\times n_A}$, $B\in\Z^{d_A\times n_B}$, $C\in\Z^{d_C\times
  n_B}$, $D\in\Z^{d_C\times n_A}$ be given matrices. Moreover, let $M$ be a
bound on the absolute values of the entries in $C$ and $D$, and let $g\in\Z_+$
be the number from Theorem~\ref{Lemma: Finite bound for SIP Graver
  bases}. Then for any $N\in\Z_+$ we have 
\begin{align*}
  &  \max\left\{\|\vev\|_1:\vev\in\Graver\left(\FourBlock[N]{C}{D}{B}{A}\right)\right\}\\
  & \quad \leq  (2(n_B+Nn_A)M)^{2^{d_C}-1}\left(\max\left\{\|\vev\|_1:\vev\in\Graver\left(\FourBlock[N]\NoBlock\NoBlock BA\right)\right\}\right)^{2^{d_C}}\\
  & \quad \leq (2(n_B+Nn_A)M)^{2^{d_C}-1}\left((n_B+Nn_A)g\right)^{2^{d_C}}.
\end{align*}
If $A$, $B$, $C$, $D$ are fixed matrices, then
$\max\left\{\|\vev\|_1:\vev\in\Graver\left(\FourBlock[N]{C}{D}{B}{A}\right)\right\}$
is bounded by $\BigO(N^{2^{d_C}+1})$, a polynomial in $N$.
\end{proposition}

\begin{proof}
  The first claim is a direct consequence of Theorem~\ref{Lemma: Finite bound
    for SIP Graver bases} and Corollary~\ref{Corollary: Recursive bound on
    increase of Graver degree} with $L=\FourBlock[N]\NoBlock\NoBlock BA$,
  $F=\FourBlock[N]CD\NoBlock\NoBlock$, and $E=\FourBlock[N]CDBA$.  
  The polynomial bound for fixed
  matrices $A$, $B$, $C$, $D$ and varying $N$ follows by observing
  that $n_A,n_B,d_C,M,g$ are constants as they depend only on the fixed
  matrices $A$, $B$, $C$, $D$.
\end{proof}

The above result has appeared before in   
\cite{hemmecke-koeppe-weismantel:4-block}; we included the proof to make the
present paper more self-contained.
We now complement it with a useful alternative bound, which is given by the
following new result. 
\begin{proposition}[Alternative length bound for 4-block IPs]
  \label{Proposition: Alternative polynomial bound for fixed ABCD on increase of Graver degree}
  Let $A\in\Z^{d_A\times n_A}$, $B\in\Z^{d_A\times n_B}$, $C\in\Z^{d_C\times
  n_B}$, $D\in\Z^{d_C\times n_A}$ be given matrices. Moreover, let $M$ be a
bound on the absolute values of the entries in $C$ and $D$, and let $g, \xi, \eta\in\Z_+$
be the numbers, depending on $A$ and $B$, from Theorem~\ref{Lemma: Finite bound for SIP Graver
  bases}. Then for any $N\in\Z_+$ we have   
\begin{align*}
  &\max\left\{\|\vev\|_1:\vev\in\Graver\left(\FourBlock[N]{C}{D}{B}{A}\right)\right\}\\
  &\quad \leq 
  \xi \cdot (N+\eta)^\eta \cdot d_C \cdot \paren{\sqrt{d_C} (n_B+Nn_A) M}^{d_C}
  \cdot (n_B+Nn_A)g.
\end{align*}
If $A$, $B$, $C$, $D$ are fixed matrices, then
$\max\left\{\|\vev\|_1:\vev\in\Graver\left(\FourBlock[N]{C}{D}{B}{A}\right)\right\}$
is bounded by $\BigO(N^{d_C+\eta})$, a polynomial in $N$.
\end{proposition}

Either of the two results implies \autoref{Theorem: Polynomial bound for fixed
  ABCD on increase of Graver degree}. 

\begin{remark}
  Comparing the two results is difficult because bounds for the finite number
  $\eta(A,B)$ are unknown.  However, one should expect that the bound of
  \autoref{Proposition: Alternative polynomial bound for fixed ABCD on
    increase of Graver degree} is better for matrices with large upper blocks
  $\FourBlock{C}{D}\NoBlock\NoBlock$, whereas the bound of
  \autoref{Proposition: Polynomial bound for fixed ABCD on increase of Graver
    degree} is better for matrices with large lower
  blocks~$\FourBlock\NoBlock\NoBlock{B}{A}$. 
\end{remark}

\begin{proof}[Proof of \autoref{Proposition: Alternative polynomial bound for fixed ABCD
    on increase of Graver degree}] 
  Let $L = \FourBlock[N]\cdot\cdot BA$ and $F = \FourBlock[N] CD\cdot\cdot =
  (C, D, \dots, D)$.  

  First of all, Theorem~\ref{Lemma: Finite bound for SIP Graver
    bases}\,(b) gives the bound
  \begin{equation}\label{eq:newbound:1}
    \|\ve v\|_1 \leq (n_B+Nn_A)g \quad\text{for}\quad \ve v\in \Graver(L),
  \end{equation}
  where $g$ is a constant that only depends on $A$ and~$B$.

  We now consider the matrix $F \cdot\Graver(L)$.  Each column of
  it is given by $$F \ve v = C\ve x + D\sum_{i=1}^N \ve y^i
  \quad\text{with}\quad \ve v = (\ve x, \ve y^1,\dots,\ve y^N) \in
  \Graver(L).$$ By Theorem~\ref{Lemma: Finite bound for SIP Graver
    bases}\,(c), there are at most $\xi = \BigO(1)$
  different vectors~$\ve x$ and for each~$\ve x$ at most $\eta = \BigO(1)$
  different vectors~$\ve y^i$. 
  We now determine the number~$\sigma$ of different sums $\ve s = \sum_{i=1}^N \ve y^i$ that can arise
  from these choices. This number is bounded by the number of weak
  compositions of $N$ into $\eta$ non-negative integer parts:
  $\sigma\leq \binom{N+\eta-1}{\eta-1} \leq (N+\eta)^\eta = \BigO(N^\eta)$.  Thus $F\Graver(L)$ has at
  most $d := \xi\cdot\sigma \leq \xi \cdot (N+\eta)^\eta = \BigO(N^\eta)$ different columns. 

  Using the bound on the entries of~$C$ and $D$, we find that the maximum
  absolute value of the entries of $F\Graver(L)$ is bounded by $(n_B+Nn_A)M$.
  
  We now determine a length bound for the elements~$\velambda$ of~$\Graver(F
  \cdot\Graver(L))$. By Corollary~\ref{Corollary: Determinant bound, aggregated}, we find that 
  \begin{align}
    \| \velambda\|_1 &\leq d \cdot d_C \cdot \paren{\sqrt{d_C} (n_B+Nn_A) M}^{d_C} \notag\\
    &\leq \xi \cdot (N+\eta)^\eta \cdot d_C \cdot \paren{\sqrt{d_C} (n_B+Nn_A) M}^{d_C}.
    \label{eq:newbound:2}
  \end{align}
  
  Combining the two bounds \eqref{eq:newbound:1} and \eqref{eq:newbound:2} using 
  Corollary~\ref{Lemma: General bound on increase of Graver degree} 
  then gives the result.
\end{proof}

\subsection{Constructing a feasible solution}

For constructing a feasible solution to the problem~$\IP$, we will use the algorithm of
Theorem~\ref{Theorem: ABCD}\,(a), first introduced
in~\cite{hemmecke-koeppe-weismantel:4-block}.  For sake of completeness, we
describe the algorithm here and thus give the proof of~\autoref{Theorem: ABCD}\,(a).

\begin{proof}[Proof of \autoref{Theorem: ABCD}\,(a)]
Let $N\in\Z_+$, $\vel,\veu\in\Z^{n_B+Nn_A}$, $\veb\in\Z^{d_C+Nd_A}$.  First,
construct an integer solution to the system
$\FourBlock[N]{C}{D}{B}{A}\vez=\veb$.  This can be done in polynomial time
using the Hermite normal form of $\FourBlock[N]{C}{D}{B}{A}$.
Then we turn it into a feasible solution (satisfying $\vel\leq\vez\leq\veu$)
by a sequence of at most $\BigO(Nd_A)$ many integer linear programs (with the
same problem matrix $\FourBlock[N]{C}{D}{B}{A}$, but with bounds $\ve{\tilde
  l},\ve{\tilde u}$ 
adjusted so that the current solution is feasible) with auxiliary objective
functions that move the components
of $\vez$ into the direction of the given original bounds~$\vel,\veu$, see
\cite{Hemmecke:2003b}. This step is similar to phase I of the simplex method in
linear programming. 

In order to solve these auxiliary integer linear programs with polynomially
many augmentation steps, we use the speed-up provided by the directed
augmentation procedure \cite{SchulzWeismantel99}.  This procedure requires us
to repeatedly find, for certain vectors $\ve c$ and $\ve d$ that it
constructs, an augmentation vector $\ve v$ with respect to the (separable
convex) piecewise linear function $h(\ve v) = \ve c^\intercal \ve
v^+ + \ve d^\intercal \ve v^-$.

Consequently, we only need to show how to find, for a given solution $\vez_0$
that is feasible for $({\rm IP})_{N,\veb,\ve{\tilde l},\ve{\tilde u},h}$, 
an augmenting Graver basis element
$\vev\in\Graver\left(\FourBlock[N]{C}{D}{B}{A}\right)$ for a separable convex
piecewise linear function $h(\ve v)$ in polynomial time in $N$ and
in the binary encoding lengths of $\vez_0$ and of $\vecc, \ved$. 

Let us now assume that we are given a solution
$\vez_0=(\vex_0,\vey^1_0,\ldots,\vey^N_0)$ that is feasible for $({\rm
  IP})_{N,\veb,\ve{\tilde l},\ve{\tilde u},h}$ and that we wish to decide
whether there exists another feasible solution $\vez_1$ with
$h(\vez_1-\vez_0)<0$.  
By \cite{Graver75,MurotaSaitoWeismantel04}, it
suffices to decide whether there exists some vector
$\vev=(\vebarx,\vebary^1,\ldots,\vebary^N)$ in the Graver basis of
$\FourBlock[N]{C}{D}{B}{A}$ such that $\vez_0+\vev$ is feasible
and $h(\vev)<0$.  
By \autoref{Proposition: Polynomial bound for
  fixed ABCD on increase of Graver degree} or \autoref{Proposition:
  Alternative polynomial bound for fixed ABCD on increase of Graver degree},
the $\ell_1$-norm of $\ve v$ is bounded polynomially in~$N$.  Thus, since
$n_B$ is constant, there is only a polynomial number of candidates for the
$\vebarx$-part of $\vev$.  Since the bounds given by \autoref{Proposition:
  Polynomial bound for fixed ABCD on increase of Graver degree} and
\autoref{Proposition: Alternative polynomial bound for fixed ABCD on increase
  of Graver degree} are effectively computable (cf.~\autoref{rem:stoch-bounds-effective}), we can actually list all
possible vectors~$\ve{\bar x}$ that satisfy these bounds.  

For each such
candidate $\vebarx$, we can find a best possible choice for
$\vebary_1,\ldots,\vebary_N$ by solving the following $N$-fold IP:
\begin{multline*}
  \min\left\{h\left(\vev\right):
    \begin{array}{r@{\;}c@{\;}l}
      \FourBlock[N]{C}{D}{B}{A} \left(\vez_0+\vev\right) & = & \veb,\\
      \ve{\tilde l} \leq \left(\vez_0+\vev\right) & \leq & \ve{\tilde u},\\ 
      \vev = (\vebarx,\vebary^1,\ldots,\vebary^N) & \in & \Z^{n_B+Nn_A}
    \end{array}
  \right\} \\ %% ,\\\\
  = \min\left\{h\begin{psmallmatrix}\ve{\bar x}\\ \vebary^1\\\vdots\\\vebary^N\end{psmallmatrix}:
    \begin{array}{r@{\;}c@{\;}l}
      \FourBlock[N]\NoBlock{D}\NoBlock{A} 
      \begin{psmallmatrix}\vebary^1\\\vdots\\\vebary^N\end{psmallmatrix} 
      & = & \veb - \FourBlock[N]{C}{D}{B}{A} \vez_0 
      - \FourBlock[N]{C}\NoBlock{B}\NoBlock \ve{\bar x},\\
      \ve{\tilde l}-\vez_0 \leq
      \begin{psmallmatrix}\ve{\bar x}\\ \vebary^1\\\vdots\\\vebary^N\end{psmallmatrix}  
      & \leq & \ve{\tilde u}-\vez_0,\\ 
      \vebary^1, \dots, \vebary^N &\in& \Z^{n_A}
    \end{array}
    \right\}
\end{multline*}
for given $\vez_0=(\vex_0,\vey^1_0,\ldots,\vey^N_0)$ and $\vebarx$. As shown
in the second line, this problem does indeed simplify to a separable
convex $N$-fold IP with 
problem matrix
$\FourBlock[N]\NoBlock{D}\NoBlock{A}$ because
$\vez_0=(\vex_0,\vey^1_0,\ldots,\vey^N_0)$ and $\vebarx$ are fixed. Since the
matrices $A$ and $D$ are fixed, each such $N$-fold IP is solvable in
polynomial time \cite{Hemmecke+Onn+Weismantel:08}. 
In fact, as shown in \cite{hemmecke-onn-romanchuk:nfold-cubic}, 
because the function~$h$ is ``2-piecewise affine'', 
this problem can be solved in time $\BigO(N^3L)$ by Graver-based dynamic
programming, where $L=\langle \ve c, \ve d, \ve{\tilde l}, \ve{\tilde u}, \ve z_0,
\ve{\bar x}\rangle$. 

If the $N$-fold IP is infeasible, there does not exist an augmenting vector
using the particular choice of $\vebarx$.  If it is feasible, let 
$\vev=(\vebarx,\vebary^1,\ldots,\vebary^N)$ be the optimal solution.
Now if we have $h(\vev)\geq 0$, then
no augmenting vector can be constructed using this particular choice of
$\vebarx$. If, on the other hand, we have $h(\vev)< 0$, then $\vev$ is a
desired augmenting vector for $\vez_0$ and we can stop. 

As we solve polynomially many polynomially solvable $N$-fold IPs, one for each
choice of~$\ve{\bar x}$, an optimality certificate or a desired augmentation
step can be computed in polynomial time and the claim follows.
\end{proof}

\subsection{Using Hochbaum--Shanthikumar's proximity results}

Hochbaum and Shanthikumar~\cite{hochbaum-shanthikumar:1990:convex-separable}
present an algorithm for nonlinear separable convex integer minimization
problems for matrices with small subdeterminants.  The algorithm is based on
the so-called proximity-scaling technique.  It is pseudo-polynomial in the
sense that the running time depends
polynomially on the absolute value of the largest subdeterminant of the
problem matrix.  
The results of the paper~\cite{hochbaum-shanthikumar:1990:convex-separable} cannot be
directly applied to our situation, since the subdeterminants of $N$-fold 4-block
decomposable matrices typically grow exponentially in~$N$.  In the following we adapt
a lemma from~\cite{hochbaum-shanthikumar:1990:convex-separable} that
establishes proximity of optimal solutions of the integer problem
and its continuous relaxation; we do not use the scaling technique, however.

We consider the separable convex integer minimization problem
\begin{equation}\label{eq:hochbaum-ip}
  \min \{\, f(\ve z) : E\ve z=\ve b,\ \ve l\leq\ve z\leq\ve u,\ \ve z\in \Z^n \,\}.
\end{equation}
\begin{theorem}[Proximity]\label{th:proximity}
  %% Let $\ell(E)$ denote an upper bound on the maximum $\ell_1$-norm of the vectors
  %% in the Graver basis of~$E$.
  Let $\ve{\hat r}$ be an optimal solution
  of the continuous relaxation of~\eqref{eq:hochbaum-ip},
  \begin{equation}\label{eq:hochbaum-cp}
    \min \{\, f(\ve r) : E\ve r=\ve b,\ \ve l\leq\ve r\leq\ve u,\ \ve r\in \R^n \,\}.
  \end{equation}
  Then there
  exists an optimal solution~$\ve z^*$
  of the integer optimization problem~\eqref{eq:hochbaum-ip} with
  \begin{displaymath}
    \| \ve{\hat r} - \ve z^* \|_\infty \leq n \cdot 
    \max\left\{\|\vev\|_\infty:\vev\in\Graver(E)\right\}.
  \end{displaymath}
\end{theorem}
We remark that we actually just need a bound on the circuits of~$E$,
which form a subset of the Graver basis of~$E$.
Hochbaum and Shanthikumar~\cite{hochbaum-shanthikumar:1990:convex-separable}
prove a version of this result where the
maximum of the absolute values of the subdeterminants of~$E$ appears on the
right-hand side.  Our proof is almost identical, but we include it here for
completeness. 

\begin{proof}
  Let $\ve{\hat z}$ be an optimal solution of the integer optimization
  problem~\eqref{eq:hochbaum-ip}.  Since $\ve{\hat z}$ is a feasible solution
  to
  the continuous relaxation, there exists a conformal (orthant-compatible) decomposition
  of $\ve{\hat r} - \ve{\hat z}$ into rational multiples of the circuits of~$E$,
  \begin{displaymath}
    \ve{\hat r} - \ve{\hat z} = \sum_{i=1}^n \alpha_i \ve u^i,\quad \alpha_i \geq 0,\
    \ve u^i \in \Circuits(E),
  \end{displaymath}
  where, due to Carath\'eodory's theorem, at most~$n$ circuits are needed.
  Then
  \begin{displaymath}
    \ve{\hat r} - \ve{\hat z} = \sum_{i=1}^n \lfloor \alpha_i \rfloor \ve u^i
    + \sum_{i=1}^n \beta_i \ve u^i,
  \end{displaymath}
  setting $\beta_i = \alpha_i - \lfloor \alpha_i \rfloor$.
  Now we define
  \begin{align*}
    \ve r^* &= \ve{\hat z} + \sum_{i=1}^n \beta_i \ve u^i, \quad\text{and}\quad
    \ve z^* = \ve{\hat z} + \sum_{i=1}^n \lfloor \alpha_i\rfloor \ve u^i.
  \end{align*}
  Since the vectors $\ve u^i$ lie in the kernel of matrix~$E$,
  both $\ve z=\ve z^*$ and $\ve z=\ve r^*$ satisfy the equation~$E\ve z=\ve b$.
  Moreover, since both $\ve{\hat r}$ and $\ve{\hat z}$ lie within the lower
  and upper bounds and the vectors $\ve u^i$ lie in the same
  orthant as $\ve{\hat r} - \ve{\hat z}$, also $\ve z^*$ and $\ve r^*$ lie
  within the lower and upper bounds.  Thus, $\ve r^*$ is a feasible
  solution to the continuous relaxation of~\eqref{eq:hochbaum-ip}.
  Since $\ve z^*$ is also an
  integer vector, it is a feasible solution to the integer optimization
  problem~\eqref{eq:hochbaum-ip}.

  We can write
  \begin{displaymath}
    \ve{\hat r} - \ve{\hat z} = [\ve r^* - \ve{\hat z}] + [\ve z^*-\ve{\hat z}].
  \end{displaymath}
  Then we use an important superadditivity property of separable convex
  functions (see \cite[Lemma 3.1]{hochbaum-shanthikumar:1990:convex-separable}
  and \cite{MurotaSaitoWeismantel04}), which gives
  \begin{equation}\label{eq:superadd}
    f(\ve{\hat r}) - f(\ve{\hat z}) \geq [f(\ve r^*) - f(\ve{\hat z})] 
    + [f(\ve z^*)-f(\ve{\hat z})],
  \end{equation}
  or, equivalently, 
  \begin{equation}
    f(\ve{\hat r}) - f(\ve r^*) \geq f(\ve z^*) - f(\ve{\hat z}).
  \end{equation}
  Since $\ve{\hat r}$ is an optimal solution to the continuous relaxation
  and $\ve r^*$ is a feasible solution to it, the left-hand side is
  nonpositive, and so $f(\ve z^*) \leq f(\ve{\hat z})$. 
  Thus, since $\ve z^*$ is a feasible solution to~\eqref{eq:hochbaum-ip}, 
  it is, in fact, another optimal solution of the integer optimization problem
  and $f(\ve z^*) = f(\ve{\hat z})$.

  We now verify the proximity of~$\ve z^*$ to~$\ve{\hat r}$.  From the
  definition of~$\ve z^*$, we immediately get
  \begin{align*}
    \bigl\| \ve{\hat r} - \ve z^* \bigr\|_\infty 
    &= \bigl\| [\ve{\hat r} - \ve{\hat z}] + [\ve{\hat z} - \ve z^*] \bigr\|_\infty \\
    &= \bigl\| \textstyle \sum_{i=1}^n \alpha_i \ve u^i - \sum_{i=1}^n
    \lfloor\alpha_i\rfloor \ve u^i \bigr\|_\infty \\
    &= \bigl\| \textstyle\sum_{i=1}^n \beta_i \ve u^i \bigr\|_\infty \\
    &\leq n\cdot \max\{\, \mathopen\| \ve u^j \mathclose\|_\infty : j=1,\dots,n \,\}\\
    &\leq n\cdot \max\left\{\,\|\vev\|_\infty:\vev\in\Graver(E)\,\right\}.
  \end{align*}
  This concludes the proof.
\end{proof}

As an immediate corollary, we obtain the following result.

\begin{corollary}\label{cor:restricted-ip-is-exact}
  Let $\epsilon\geq0$ and let $\ve{\hat r}$ be an optimal solution
  to the continuous relaxation~\eqref{eq:hochbaum-cp}.
  Setting
  \begin{align*}
    \ve l' &= \max\{\ve l, \floor{\ve{\hat r} - (n\cdot
      \ell)\ve1}\},\\
    \ve u' &= \min\{\ve u, \ceil{\ve{\hat r} +
      (n\cdot \ell)\ve1}\},
  \end{align*}
  where $ \ell = \max\left\{\,\|\vev\|_\infty:\vev\in\Graver(E)\,\right\} $,
  we have 
  \begin{multline}
    \min \{\, f(\ve z) : E\ve z=\ve b,\ \ve l\leq\ve z\leq\ve u,\ \ve z\in
    \Z^n \,\}\\
    = \min \{\, f(\ve z) : E\ve z=\ve b,\ \ve l'\leq\ve z\leq\ve u',\ \ve z\in
    \Z^n \,\}.
  \end{multline}
\end{corollary}
Later we will use a simple modification of
\autoref{cor:restricted-ip-is-exact}, using an $\epsilon$-approximate optimal
solution to the continuous relaxation~\eqref{eq:hochbaum-cp}.

For $E = \FourBlock[N]{C}{D}{B}{A}$, we can control the size of $\ell$ using
\autoref{Proposition: Polynomial bound for fixed ABCD on increase of Graver
  degree} or \autoref{Proposition: Alternative polynomial bound for fixed ABCD
  on increase of Graver degree} and thus obtain an equivalent IP with small
(polynomial-sized) bounds.  

We note that though the bounds are small, the dimension is still variable, and
so the problem cannot be solved efficiently with elementary techniques such as
dynamic programming.  In the following subsections, we show how to solve this
IP with Graver basis techniques.

\subsection{Graver-best augmentation for the restricted problem}

In the restricted problem, no long augmentation steps are possible, and
therefore it is possible to efficiently construct a Graver-best augmentation
vector.  Using this observation, we prove the following theorem.

\begin{theorem}\label{Theorem: Restricted optimization over 4-block N-fold matrix in poly-time}
  Let $A\in\Z^{d_A\times n_A}$, $B\in\Z^{d_A\times n_B}$, $C\in\Z^{d_C\times
    n_B}$, $D\in\Z^{d_C\times n_A}$ be fixed matrices. Then there exists an
  algorithm that, given $N\in\Z_+$, 
  $\vecc\in\Z^{kn_B+kNn_A}$, $\veb\in\Z^{d_C+Nd_A}$,
  $\vel,\veu\in\Z^{n_B+Nn_A}$, a feasible
  solution~$\ve z_0$, and a comparison oracle for the
  function~$f\colon\R^{n_B+Nn_A}\rightarrow\R$,  
  finds an optimal solution to 
  $$\min\left\{f(\vez):\FourBlock[N]{C}{D}{B}{A}\vez=\veb,\;\vel'\leq\vez\leq\veu',\;\vez\in\Z^{n_B+Nn_A}\right\}$$
  and that runs in time that is polynomially bounded in $N$, in $k := \| \ve u'
  - \ve l' \|_\infty$, 
  and in the binary encoding lengths $\langle \veb,\vecc, \hat f\rangle$.
\end{theorem}
\begin{proof}
  By the Graver-best speed-up technique
  \cite{Hemmecke+Onn+Weismantel:08}, it suffices to show that for a 
  given feasible solution~$\ve z_0$, we can construct a vector $\gamma\ve g$,
  where $\gamma \in \Z_+$ and $\ve
  g\in\Graver\left(\FourBlock[N]{C}{D}{B}{A}\right)$, such that $\ve z_0 +
  \gamma\ve g$ is feasible, and $\gamma$ and $\ve g$ minimize $f(\ve
  z_0+\gamma\ve g)$ among all possible choices.  It actually suffices to
  construct any vector $\ve v$ such that $\ve z_0+\ve v$ is feasible and $f(\ve z_0+\ve
  v) \leq f(\ve z_0+\gamma\ve g)$. 
  
  Write $\ve z_0=(\vex_0,\vey^1_0,\ldots,\vey^N_0)$ and let
  $\vev=(\vebarx,\dots)$ be any vector in the 
  Graver basis of 
  $\FourBlock[N]{C}{D}{B}{A}$.
  By \autoref{Proposition: Polynomial bound for
    fixed ABCD on increase of Graver degree} or \autoref{Proposition:
    Alternative polynomial bound for fixed ABCD on increase of Graver degree},
  the $\ell_1$-norm of $\ve v$ is bounded polynomially in~$N$.  Thus, since
  $n_B$ is constant, there is only a polynomial number of candidates for the
  $\vebarx$-part of $\vev$.  Since the bounds given by \autoref{Proposition:
    Polynomial bound for fixed ABCD on increase of Graver degree} and
  \autoref{Proposition: Alternative polynomial bound for fixed ABCD on increase
    of Graver degree} are effectively computable
  (cf.~\autoref{rem:stoch-bounds-effective}), we can actually list all 
  possible vectors~$\ve{\bar x}$ that satisfy these bounds.  

  For each such vector~$\ve{\bar x}$, we now consider all vectors of
  the form $(\gamma\ve{\bar x}, \vebary^1,\dots,\vebary^N)$ as candidate
  augmentation vectors, not just multiples $\gamma\ve v$ of Graver basis
  elements.

  In the special case $\ve{\bar x} = \ve 0$, this is
  equivalent to the construction of a Graver-best augmentation vector for
  the $N$-fold IP with the problem
  matrix~$\FourBlock[N]\NoBlock{D}\NoBlock{A}$, 
  which can be done in polynomial time \cite{Hemmecke+Onn+Weismantel:08}.

  Otherwise, if $\ve{\bar x} \neq \ve0$, we determine the largest step
  length $\hat\gamma\in\Z_+$ such that $\vex_0 + \hat\gamma\ve{\bar x}$
  lies within the bounds~$\ve l', \ve u'$.  Certainly $\hat\gamma \leq k$. 
  We now check each possible step length $\gamma = 1, 2, \dots, \hat\gamma$
  separately.  To find a best possible choice for
  $\vebary_1,\ldots,\vebary_N$, we solve the following $N$-fold IP:
  \begin{displaymath}
  \min\left\{f\left(\vev\right):
    \begin{array}{r@{\;}c@{\;}l}
      \FourBlock[N]{C}{D}{B}{A} \left(\vez_0+\vev\right) & = & \veb,\\
      \ve{l}' \leq \left(\vez_0+\vev\right) & \leq & \ve{u}',\\ 
      \vev = (\gamma\vebarx,\vebary^1,\ldots,\vebary^N) & \in & \Z^{n_B+Nn_A}
    \end{array}
  \right\}. 
\end{displaymath}
Since the
matrices $A$ and $D$ are fixed, each such $N$-fold IP is solvable in
polynomial time \cite{Hemmecke+Onn+Weismantel:08}. 

If the $N$-fold IP is infeasible, there does not exist an augmenting vector
using the particular choice of $\vebarx$ and $\gamma$.  If it is feasible, let 
$\vev=(\gamma\vebarx,\vebary^1,\ldots,\vebary^N)$ be an optimal solution.
Now if we have $f(\vev)\geq 0$, then
no augmenting vector can be constructed using this particular choice of
$\vebarx$ and $\gamma$. 
If, on the other hand, we have $f(\vev)< 0$, then $\vev$ is a
candidate for the Graver-best augmentation vector.

By iterating over all $\ve{\bar x}$ and all $\gamma$, we efficiently construct
a Graver-best augmentation vector.  
\end{proof}

\begin{remark} A more precise complexity analysis is as follows. 
  \begin{enumerate}[\rm(a)]
  \item 
    For the construction in the special case $\ve{\bar x} = \ve 0$: In fact,
    by \cite[Lemma 3.4 and proof of Theorem
    4.2]{hemmecke-onn-romanchuk:nfold-cubic}, for any of the possible step
    lengths $\gamma = 1,2,\dots,k$, we can find in linear time $\BigO(N)$ an
    augmenting vector $\gamma\ve v$ that is at least as good as the best
    Graver step $\gamma\ve g$ with $\ve
    g\in\Graver\FourBlock[N]\NoBlock{D}\NoBlock{A}$.  Checking all step
    lengths, we get a complexity of $\BigO(kN)$.
  \item For the solution of the $N$-fold subproblem in the general case
    $\ve{\bar x} \neq\ve0$: 
    This optimization, in turn, uses another Graver-best augmentation
    technique.  In Phase I, the possible step lengths are large, but the
    auxiliary objective functions are linear, and so the running time is
    $\BigO(N^3L)$ by Graver-based dynamic programming \cite[Theorem
    3.9]{hemmecke-onn-romanchuk:nfold-cubic}, where $L=\langle\ve{l}'$,
    $\ve{u}'$, $\ve z_0$, $\ve{\bar x}\rangle$. 
    In Phase II, there are few possible step lengths, $\gamma = 1,2,\dots,k$,
    so we can try them all.  By \cite[Lemma 3.4 and proof of Theorem
    4.2]{hemmecke-onn-romanchuk:nfold-cubic}, we can find for a fixed~$\gamma$
    in linear time $\BigO(N)$ an augmenting vector $\gamma\ve v$ that is at
    least as good as the best Graver step $\gamma\ve g$ with $\ve
    g\in\Graver\FourBlock[N]\NoBlock{D}\NoBlock{A}$.  Checking all step
    lengths, we get a complexity of $\BigO(kN)$.  Using the results of
    \cite{Hemmecke+Onn+Weismantel:08} (modified with the optimality
    criterion of \cite{MurotaSaitoWeismantel04}), the number of Graver-best
    augmentations is bounded by $\BigO(N\langle \hat f\rangle)$.  Thus the
    complexity of this subproblem is $\BigO(N^2k\langle\hat f\rangle + N^3L)$.
  \item 
    The number of steps in the overall Graver-best augmentation algorithm for the
    restricted 4-block decomposable problem is
    again bounded by $\BigO(N\langle \hat f\rangle)$.  
  \end{enumerate}
\end{remark}

\begin{remark}
  Other augmentation techniques can be used to prove \autoref{Theorem: Restricted optimization over 4-block N-fold matrix in poly-time}.
  For example, following \cite[section
  2]{hochbaum-shanthikumar:1990:convex-separable}, we can reformulate a
  separable convex integer minimization problem with small bounds as a 0/1
  linear integer minimization problem in the straightforward way.  Then we can
  apply the bit-scaling speed-up technique, for instance
  \cite{SchulzWeismantelZiegler95}.
\end{remark}

\subsection{Putting all together}

For each set of fixed matrices $A$, $B$, $C$, $D$
and for any function $\epsilon(N)$ that is bounded polynomially in~$N$,
we consider the following algorithm.
\smallbreak

\begin{algorithm}[Graver proximity algorithm]\label{algo:graver-proximity}~
  \begin{algorithmic}
    \INPUT $N\in\Z_+$, bounds $\vel,\veu\in\Z^{n_B+Nn_A}$, right-hand side
    $\veb\in\Z^{d_C+Nd_A}$, evaluation oracle for a separable convex function
    $f\colon\R^{n_B+Nn_A}\rightarrow\R$, approximate continuous convex
    optimization oracle. %% an upper bound~$\hat f$ on $|f|$ over
    %% the feasible region of~$\IP$ 
    \OUTPUT 
    an optimal solution~$\ve z^*$ to $\IP$ or \textsc{Infeasible} or
    \textsc{Unbounded}.
    \STATE
    Let $n = n_B+Nn_A$ denote the dimension of the problem.
    \STATE
    Call the approximate continuous convex optimization oracle with
    $\epsilon=\epsilon(N)$ to 
    find an approximate solution~$\ve r_\epsilon\in\Q^{n_B+Nn_A}$ to the
    continuous relaxation  
    \[
    \min\left\{f(\ver):\FourBlock[N]{C}{D}{B}{A}\ver=\veb,\;\vel\leq\ver\leq\veu,\;\ver\in\R^{n_B+Nn_A}\right\}.%  
    \]%
    \IF{oracle returns \textsc{Infeasible}}
    \RETURN{\textsc{Infeasible}}.
    \ELSIF{oracle returns \textsc{Unbounded}}
    \RETURN{\textsc{Unbounded}}.
    \ELSE
    \STATE Compute an upper bound $\ell$ on the maximum $\ell_1$-norm of the vectors
    in $\Graver\paren{\FourBlock[N]{C}{D}{B}{A}}$, using
    \autoref{Proposition: Polynomial bound for fixed ABCD on increase of
      Graver degree} or \autoref{Proposition: Alternative polynomial bound for fixed ABCD on increase of Graver degree}. 
    \STATE Let $\ve l' = \max\{\ve l, \floor{\ve r_\epsilon - (n\cdot
      \ell+\epsilon)\ve1}\}$ and $\ve u' = \min\{\ve u,
    \ceil{\ve r_\epsilon + (n\cdot \ell+\epsilon)\ve1}\}$. 
    \STATE Let $k = \|\ve u'-\ve l'\|_\infty$. 
    \STATE Using the algorithm of \autoref{Theorem: ABCD}\,(a), find a
    feasible solution~$\ve z_0$ for the restricted convex integer 
    minimization problem 
    $$\min\left\{f(\vez):\FourBlock[N]{C}{D}{B}{A}\vez=\veb,\;\vel'\leq\vez\leq\veu',\;\vez\in\Z^{n_B+Nn_A}\right\}.$$
    %%%%%%%%%%%%%%%%%%%%%%%%%%%% GREEDY GRAVER-VARIANTE: %%%%%%%%%%%%%%%%%%%%
    \STATE Solve the problem to optimality using the algorithm of
    \autoref{Theorem: Restricted optimization over 4-block N-fold matrix in
      poly-time}.  
    \ENDIF
  \end{algorithmic}
\end{algorithm}

By analyzing this algorithm, we now prove the main theorem of this paper.

\begin{proof}[Proof of \autoref{Theorem: ABCD convex}]
  We first show that \autoref{algo:graver-proximity} is correct. 
  If the continuous relaxation $\CP$ is infeasible or unbounded, then so is
  the problem $\IP$.  In the following, assume that $\CP$ has an optimal
  solution.  Then there exists an optimal solution~$\ve{\hat r}$ to $\CP$
  with $\|\ve{\hat r} - \ve{r}_\epsilon\|_\infty \leq \epsilon$.  By
  \autoref{th:proximity}, there
  exists an optimal solution~$\ve z^*$
  of the integer optimization problem~$\IP$ with
  \begin{math}
    \| \ve{\hat r} - \ve z^* \|_\infty \leq n \cdot \ell.
  \end{math}
  By the triangle inequality, this solution then satisfies $\| \ve z^* -
  \ve{r}_\epsilon \|_\infty \leq n\cdot \ell + \epsilon$ and is therefore a
  feasible solution to the restricted IP with variable bounds $\ve l'$ and
  $\ve u'$.  Thus it suffices to solve the restricted IP to optimality, 
  %%%%%%%%%%%%%%%%%%%%%%%%%%%% GREEDY GRAVER-VARIANTE: %%%%%%%%%%%%%%%%%%%%
  which is done with the algorithm of \autoref{Theorem: Restricted
    optimization over 4-block N-fold matrix in poly-time}. 
    
  The algorithm has the claimed complexity because $$k \leq 2((n_B+Nn_A)\cdot\ell +
  \epsilon)$$ is bounded polynomially in $N$ by \autoref{Proposition:
    Polynomial bound for fixed ABCD on increase of Graver degree} or
  \autoref{Proposition: Alternative polynomial bound for fixed ABCD on
    increase of Graver degree}.
  The complexity then follows from \autoref{Theorem: Restricted optimization
    over 4-block N-fold matrix in poly-time}.  
\end{proof}

%\clearpage
\subsection*{Acknowledgments.} We wish to thank R\"udiger Schultz for valuable
comments and for pointing us to \cite{Gollmer+Gotzes+Schultz}.  We also would
like to thank Shmuel Onn for pointing us toward the paper by Hochbaum and
Shanthikumar.  The second author was supported by grant DMS-0914873 of the
National Science Foundation.  A part of this work was completed during a stay
of the three authors at BIRS.

We %would like to 
dedicate this paper to the memory of Uri Rothblum.  His paper
\cite{Onn+Rothblum} has been an inspiration for our work on nonlinear
discrete optimization.  As a coauthor of R.H.\ and R.W.\ in
\cite{DeLoera+Hemmecke+Onn+Rothblum+Weismantel:09}, Uri contributed to the
application of Graver basis techniques for block-structured problems.  We
believe that the present paper continues the theme of research at the
interface of algebra, geometry, combinatorics, and optimization that Uri
appreciated.

\bibliography{/home/mkoeppe/w/siam-book/mkoeppe-bib/nl-and-eng,/home/mkoeppe/w/siam-book/mkoeppe-bib/hemmecke,/home/mkoeppe/w/siam-book/mkoeppe-bib/weismantel,/home/mkoeppe/w/siam-book/book}
\bibliographystyle{amsabbrv}

\end{document}